\documentclass[11pt]{amsart}

\usepackage[T1]{fontenc}
\usepackage{amssymb}
\usepackage{enumitem}
\usepackage[margin=1in]{geometry}
\usepackage[hidelinks]{hyperref}
\hypersetup{
  pdftitle={Soft normed linear spaces via soft elements: decomposition, convergence, and operators},
  pdfauthor={Subhasis Ray},
  pdfsubject={Soft normed linear spaces and soft linear operators},
  pdfkeywords={soft set, soft element, soft normed linear space, selection space, box topology, soft linear operator}
}

\newcommand{\SE}{\operatorname{SE}}
\newcommand{\K}{\mathbb{K}}
\newcommand{\R}{\mathbb{R}}
\newcommand{\N}{\mathbb{N}}
\newcommand{\RA}{\R^{A}}
\newcommand{\KA}{\K^{A}}
\newcommand{\Pw}{\mathcal{P}}
\newcommand{\sle}{\mathrel{\widetilde{\leq}}}
\newcommand{\slt}{\mathrel{\widetilde{<}}}
\newcommand{\sge}{\mathrel{\widetilde{\geq}}}
\newcommand{\ins}{\mathrel{\in_{s}}}
\newcommand{\subs}{\mathrel{\subseteq_{s}}}
\newcommand{\zbar}{\bar{0}}
\newcommand{\snorm}{N}
\newcommand{\bnd}{\mathrm{b}}
\newcommand{\taubox}{\tau_{\Box}}
\newcommand{\tauprod}{\tau_{\Pi}}
\newcommand{\tauunif}{\tau_{\infty}}

\setlist[itemize]{leftmargin=2.2em,itemsep=2pt,topsep=4pt}
\setlist[enumerate]{leftmargin=2.2em,itemsep=2pt,topsep=4pt}

\theoremstyle{definition}
\newtheorem{definition}{Definition}[section]
\newtheorem{example}[definition]{Example}
\newtheorem{remark}[definition]{Remark}

\theoremstyle{plain}
\newtheorem{theorem}[definition]{Theorem}

\newtheorem{proposition}[definition]{Proposition}
\newtheorem{corollary}[definition]{Corollary}
\numberwithin{equation}{section}

\title[Soft normed linear spaces via soft elements]
{Soft normed linear spaces via soft elements: decomposition, convergence, and operators}

\author{Subhasis Ray}
\address{Department of Mathematics, Siksha-Bhavana, Visva-Bharati, Santiniketan 731235, West Bengal, India}
\email{subhasis.ray@visva-bharati.ac.in}
\thanks{ORCID: 0000-0003-1100-4632}
\date{August 2026}

\begin{document}

\begin{abstract}
Let $F\colon A\to\Pw(X)$ be a soft linear subspace of a linear space $X$ over
$\K$. Its soft elements form the product $\prod_{t\in A}F(t)$. We use this
representation to characterize selection spaces, decompose soft norms and soft
linear operators into their sectionwise parts, and compare soft, uniform, and
componentwise convergence. We also characterize completeness and continuity in
the box, extended-supremum, and product settings. Neumann-series and fixed-point
results are obtained as applications.
\end{abstract}

\maketitle

\section{Introduction}

Molodtsov introduced soft set theory in \cite{Molodtsov}, and Maji, Biswas, and
Roy developed its early theory in \cite{Maji}. Soft topological spaces were
introduced in \cite{Cagman,ShabirNaz}; a recent survey is given in
\cite{Alcantud}. Das and Samanta studied soft real numbers and soft elements
\cite{DSreal}. These ideas were later used in soft metric spaces
\cite{DSmetric}, soft linear and soft normed linear spaces \cite{DMS}, and soft
linear operators \cite{DSoperators}. Other forms of soft normed spaces and soft
Banach algebras appear in \cite{Yazar,ThakurSamanta}.

A soft element of $F$ chooses one point from each section $F(t)$, so
\[
   \SE(F)=\prod_{t\in A}F(t).
\]
The soft-element viewpoint was developed for soft sets, groups, topology,
rings and ideals, and measure-theoretic structures
\cite{RayGoldar2017,GoldarRayTopology2017,GoldarRay2019,GoldarRayRing,
GoldarRayTopGroup,GoldarRayMeasure,GoldarRayLebesgue}. Related work treats soft
ideals in soft $\sigma$-topologies, filters and nets, topological modules,
bitopological spaces and groups, uniform spaces, and metric spaces
\cite{GoldarRaySarkarIdeals,GoldarRaySarkarFilters,RayModules,RayBitop,
RayBitopGroups,RayUniform,RayMetric}.

For each $t\in A$, the coordinate idempotent $e_t=\chi_{\{t\}}$ isolates the
$t$-th section. The identities $N(e_ta)=e_tN(a)$ and
$T(e_ta)=e_tT(a)$ show that both a soft norm and a soft linear operator act
sectionwise. The coordinate-independence conditions used in
\cite{DMS,DSoperators} therefore follow from the standard soft-scalar axioms.

The first result characterizes selection spaces by closure under arbitrary
coordinatewise patching. For soft linear operators, sectionwise boundedness is
equivalent to box continuity, whereas continuity for the extended-supremum
topology requires one bound valid for all sections.

The related construction in \cite{RayMetric} uses bounded section metrics. Here
untruncated norms give an extended metric on the full selection space and a
supremum norm on the bounded selections.

\section{Preliminaries}

 Throughout, $\K\in\{\R,\mathbb C\}$ and
$\N=\{1,2,\ldots\}$. Let $A$ be a nonempty parameter set and
let $X$ be a nonempty set. A \emph{soft set} over $X$ is a map
$F\colon A\to\Pw(X)$. Assume that $F(t)\ne\varnothing$ for every $t\in A$.
A \emph{soft element} of $F$ is a map $a\colon A\to X$ such that
$a(t)\in F(t)$ for every $t\in A$. We write $a\ins F$ and
\[
   \SE(F):=\{a\colon A\to X:a(t)\in F(t)\text{ for every }t\in A\}
          =\prod_{t\in A}F(t).
\]
For soft sets $H$ and $F$ over the same universe, $H\subs F$ means
$H(t)\subseteq F(t)$ for every $t\in A$.

A \emph{soft scalar} is a map $r\colon A\to\K$. The set of soft scalars is
$\KA$, with pointwise operations. For $k\in\K$, write $\bar{k}$ for the constant
soft scalar with value $k$. For $B\subseteq A$, let $\chi_B$ be its
characteristic soft scalar, and set
\[
       e_t:=\chi_{\{t\}}\qquad(t\in A).
\]
Each $\chi_B$ is an idempotent of $\KA$. A soft real number is a map from $A$ to
$\R$, that is, an element of $\RA$. We use the pointwise order
\[
 r\sle s \iff r(t)\le s(t)\quad(t\in A),
 \qquad
 r\slt s \iff r(t)<s(t)\quad(t\in A).
\]
A soft real $\varepsilon$ is \emph{strictly positive} if
$\varepsilon(t)>0$ for every $t\in A$. The zero soft real is denoted by
$\zbar$.

From Section~3 onward, $X$ is a linear space over $\K$. A soft set $F$ over
$X$ is a \emph{soft linear subspace} if each $F(t)$ is a linear subspace of
$X$. We allow the sections to vary with $t$. The absolute case used in
\cite{DMS} is obtained by taking $F(t)=X$ for every $t$. The selection space
$\SE(F)$ is then a vector space under pointwise operations and a $\KA$-module
under
\[
    (ra)(t)=r(t)a(t) \qquad(r\in\KA,\ a\in\SE(F)).
\]
Its zero element is denoted by $\theta$. For $t\in A$ and $\alpha\in F(t)$,
define
\[
 \iota_t(\alpha)(s)=
 \begin{cases}
 \alpha,&s=t,\\
 0,&s\ne t.
 \end{cases}
\]
Then $e_ta=\iota_t(a(t))$ for every $a\in\SE(F)$.

\section{Selection spaces and patchwork representation}

\begin{proposition}
If $F$ is a soft linear subspace of $X$, then $\SE(F)$ is a linear subspace of
$X^{A}$ and a module over $\KA$. In particular, $e_ta\in\SE(F)$ and
$\iota_t(\alpha)\in\SE(F)$ for every $a\in\SE(F)$, $t\in A$, and
$\alpha\in F(t)$.
\end{proposition}

\begin{proof}
Let $a,b\in\SE(F)$, $k\in\K$, and $r\in\K^A$. For every $t\in A$,
\[
 (a+b)(t)=a(t)+b(t)\in F(t),\qquad
 (ka)(t)=k\,a(t)\in F(t),
\]
because $F(t)$ is a linear subspace. Thus $\SE(F)$ is a linear subspace of
$X^A$. Also,
\[
       (ra)(t)=r(t)a(t)\in F(t),
\]
so $\SE(F)$ is a $\K^A$-module. Finally,
$e_ta=\iota_t(a(t))\in\SE(F)$, and
$\iota_t(\alpha)\in\SE(F)$ for every $\alpha\in F(t)$.
\end{proof}

We now identify the subsets of $X^A$ that are selection spaces.

\begin{definition}[Patchwork]
Let $W\subseteq X^{A}$ be nonempty. For an $A$-indexed family
$(a_t)_{t\in A}\subseteq W$, its \emph{patch} is the map $p\in X^{A}$ defined by
$p(t)=a_t(t)$. The set $W$ is \emph{patchwork-closed} if every such patch belongs
to $W$.
\end{definition}

\begin{theorem}\label{thm:3.3}
Let $W\subseteq X^{A}$ be nonempty. The following are equivalent.
\begin{enumerate}[label=\textup{(\roman*)}]
\item $W=\SE(F)$ for a soft set $F$ with nonempty sections.
\item $W$ is patchwork-closed.
\end{enumerate}
When these conditions hold, $F$ is unique and
\[
       F(t)=\pi_t(W):=\{a(t):a\in W\} \qquad(t\in A).
\]
If, in addition, $W$ is a linear subspace of $X^{A}$, then $F$ is a soft linear
subspace.
\end{theorem}

\begin{proof}
Suppose first that $W=\SE(F)$. If $(a_t)_{t\in A}\subseteq W$ and
$p(t)=a_t(t)$, then $p(t)\in F(t)$ for every $t$, so $p\in\SE(F)=W$.

Conversely, assume that $W$ is patchwork-closed and define
$F(t)=\pi_t(W)$. Then $W\subseteq\SE(F)$. Let $c\in\SE(F)$. For each $t$, choose
$a_t\in W$ with $a_t(t)=c(t)$. The patch of $(a_t)_{t\in A}$ is $c$, so
$c\in W$. Hence $W=\SE(F)$.

For uniqueness, suppose that $W=\SE(G)$ and every $G(t)$ is nonempty. Fix
$t\in A$. Since every member of $W$ has its $t$-th coordinate in $G(t)$,
$\pi_t(W)\subseteq G(t)$. Conversely, choose $b\in W$ and $\gamma\in G(t)$.
Replace $b(t)$ by $\gamma$ and leave the other coordinates unchanged. The
resulting selection belongs to $\SE(G)=W$, so $\gamma\in\pi_t(W)$. Therefore
$G(t)=\pi_t(W)$.

If $W$ is a linear subspace, then each projection $\pi_t(W)$ is also a
linear subspace of $X$.
\end{proof}

\begin{remark}
Ray and Goldar \cite{RayGoldar2017} represent a soft set by its soft
elements. Theorem~\ref{thm:3.3} gives the converse. When $A$ is
infinite, the proof uses the axiom of choice to select $(a_t)_{t\in A}$.
\end{remark}

\begin{example}
Let $X=\R$.
\begin{enumerate}[label=\textup{(\alph*)}]
\item For $A=\{1,2\}$, the diagonal
$\Delta=\{(x,x):x\in\R\}$ is a linear subspace of $\R^{2}$ but is not
patchwork-closed. Hence $\Delta\ne\SE(F)$ for every soft set $F$.
\item For $A=\N$, the space $c_0\subseteq\R^{\N}$ is closed under every binary
coordinatewise mixture $\chi_Ba+\chi_{\N\setminus B}b$, but it is not
patchwork-closed. Indeed, the vectors $u_t=\iota_t(1)$ belong to $c_0$, whereas
their patch is the constant sequence $1$, which does not belong to $c_0$.
\end{enumerate}
\end{example}

\section{Soft norms and sectionwise decomposition}

\begin{definition}[Soft norm]
Let $F$ be a soft linear subspace of $X$. A \emph{soft norm} on $F$ is a map
$\snorm\colon\SE(F)\to\RA$ satisfying, for all $a,b\in\SE(F)$ and
$r\in\KA$,
\begin{enumerate}[label=\textup{(SN\arabic*)}]
\item $\snorm(a)\sge\zbar$;
\item $\snorm(a)=\zbar$ if and only if $a=\theta$;
\item $\snorm(ra)=|r|\,\snorm(a)$;
\item $\snorm(a+b)\sle\snorm(a)+\snorm(b)$.
\end{enumerate}
The pair $(F,\snorm)$ is called a \emph{soft normed linear space via soft
elements}.
\end{definition}

\begin{example}[Sectionwise soft norms]\label{ex:4.2}
For each $t\in A$, let $\lVert\cdot\rVert_t$ be a norm on $F(t)$. Then
\[
        \snorm(a)(t)=\lVert a(t)\rVert_t
        \qquad(a\in\SE(F),\ t\in A)
\]
defines a soft norm. In particular, a norm on the ambient space $X$ induces a
soft norm by restriction to the sections.
\end{example}

\begin{example}[Constant homogeneity is insufficient]\label{ex:4.3}
Let $A=\{1,2\}$ and $F(1)=F(2)=\R$. Define
\[
 M(a)=\bigl(|a(1)|+|a(2)|,\ |a(2)|\bigr),\qquad a\in\R^2.
\]
The map $M$ is nonnegative, definite, subadditive, and homogeneous for
every constant scalar. It is not homogeneous for all soft scalars: with $r=e_1$,
\[
 M(ra)=\bigl(|a(1)|,0\bigr),\qquad
 |r|M(a)=\bigl(|a(1)|+|a(2)|,0\bigr).
\]
Thus constant homogeneity does not force a coordinatewise norm.
\end{example}

\begin{theorem}\label{thm:4.4}
Let $(F,\snorm)$ be a soft normed linear space. There is a unique family of norms
$\{\lVert\cdot\rVert_t:t\in A\}$, with $\lVert\cdot\rVert_t$ defined on $F(t)$,
such that
\[
       \snorm(a)(t)=\lVert a(t)\rVert_t
       \qquad(a\in\SE(F),\ t\in A).
\]
\end{theorem}

\begin{proof}
Fix $t\in A$. If $a(t)=b(t)$, then $e_ta=e_tb$. By (SN3),
\[
 e_t\snorm(a)=\snorm(e_ta)=\snorm(e_tb)=e_t\snorm(b),
\]
and evaluation at $t$ gives $\snorm(a)(t)=\snorm(b)(t)$. Hence
\[
  \lVert\alpha\rVert_t:=\snorm\bigl(\iota_t(\alpha)\bigr)(t),
  \qquad \alpha\in F(t),
\]
is well defined and satisfies $\snorm(a)(t)=\lVert a(t)\rVert_t$.

Axiom (SN1) gives nonnegativity. Suppose $\lVert\alpha\rVert_t=0$ and put
$a=\iota_t(\alpha)$. Since $a=e_ta$, (SN3) gives
$\snorm(a)=e_t\snorm(a)$. Its $t$-th coordinate is zero by assumption and all
other coordinates are zero because of the factor $e_t$. Thus
$\snorm(a)=\zbar$, so $a=\theta$ by (SN2), and $\alpha=0$.
Applying (SN3) to the constant soft scalar $\bar{k}$ gives homogeneity,
and evaluating (SN4) at $t$ gives the triangle inequality. Therefore
$\lVert\cdot\rVert_t$ is a norm. Uniqueness follows from
$\lVert\alpha\rVert_t=\snorm(\iota_t(\alpha))(t)$.
\end{proof}

\begin{corollary}
For every $t\in A$ and $\alpha\in F(t)$, the set
\[
 \{\snorm(a)(t):a\in\SE(F),\ a(t)=\alpha\}
\]
is a singleton.
\end{corollary}

\begin{proof}
By Theorem~\ref{thm:4.4}, the displayed set is
$\{\lVert\alpha\rVert_t\}$. This is condition (N5) in
\cite[Theorem~5.5]{DMS}, so (N5) follows from (SN1)--(SN4).
\end{proof}

\begin{remark}
Theorem~\ref{thm:4.4} and
Example~\ref{ex:4.2} give a one-to-one correspondence between soft
norms on $F$ and families of norms on the sections $F(t)$.
\end{remark}

From now on, $\lVert\cdot\rVert_t$ always denotes the norm determined by
Theorem~\ref{thm:4.4}.

\section{Bounded selections and the extended supremum metric}

Let
\[
       A_0:=\{t\in A:F(t)\ne\{0\}\}
\]
be the set of active parameters.

\begin{definition}
The space of \emph{bounded soft elements} is
\[
 \SE_{\bnd}(F):=
 \left\{a\in\SE(F):\lVert a\rVert_\infty:=
       \sup_{t\in A}\lVert a(t)\rVert_t<\infty\right\}.
\]
For arbitrary $a,b\in\SE(F)$, define
\[
       d_\infty(a,b):=\sup_{t\in A}\lVert a(t)-b(t)\rVert_t
       \in[0,\infty].
\]
\end{definition}

\begin{proposition}
Let $(F,\snorm)$ be a soft normed linear space.
\begin{enumerate}[label=\textup{(\alph*)}]
\item $\SE_{\bnd}(F)$ is a linear subspace of $\SE(F)$,
      $\lVert\cdot\rVert_\infty$ is a norm on it, and bounded soft scalars act on
      it.
\item $d_\infty$ is an extended metric on $\SE(F)$. Its finite-distance
      components are the affine sets $a+\SE_{\bnd}(F)$, each isometric to
      $(\SE_{\bnd}(F),\lVert\cdot\rVert_\infty)$.
\item If $A_0$ is finite, then $\SE_{\bnd}(F)=\SE(F)$. If $F(t)=X$ for every $t$
      and the same norm is used on all sections, then
      $\SE_{\bnd}(F)=\ell^\infty(A,X)$.
\end{enumerate}
\end{proposition}

\begin{proof}
If $a,b\in\SE_{\bnd}(F)$ and $k\in\K$, then
\[
 \lVert a+b\rVert_\infty
 \le \lVert a\rVert_\infty+\lVert b\rVert_\infty,
 \qquad
 \lVert ka\rVert_\infty=|k|\lVert a\rVert_\infty.
\]
Hence $\SE_{\bnd}(F)$ is a linear subspace. If $r\in\K^A$ is bounded, then
\[
  \lVert ra\rVert_\infty
  \le \left(\sup_{t\in A}|r(t)|\right)\lVert a\rVert_\infty,
\]
so bounded soft scalars act on this space. The norm is definite because
$\lVert a\rVert_\infty=0$ implies $\snorm(a)=\zbar$ and hence $a=\theta$.

Symmetry and separation for $d_\infty$ follow coordinatewise. For the triangle
inequality,
\[
\begin{aligned}
 d_\infty(a,c)
 &=\sup_{t\in A}\lVert a(t)-c(t)\rVert_t\\
 &\le \sup_{t\in A}
    \bigl(\lVert a(t)-b(t)\rVert_t+\lVert b(t)-c(t)\rVert_t\bigr)\\
 &\le d_\infty(a,b)+d_\infty(b,c).
\end{aligned}
\]
Thus $d_\infty$ is an extended metric. Also,
$d_\infty(a,b)<\infty$ exactly when $b-a\in\SE_{\bnd}(F)$. Its
finite-distance component through $a$ is therefore
$a+\SE_{\bnd}(F)$, and translation by $-a$ is an isometry onto
$\SE_{\bnd}(F)$.

If $A_0$ is finite, only finitely many coordinate norms can be nonzero, so every
selection is bounded. If all sections equal $X$ and carry the same norm, the
definition of $\SE_{\bnd}(F)$ is exactly the definition of
$\ell^\infty(A,X)$.
\end{proof}

\begin{remark}
For infinite $A_0$, boundedness is generally not a sectionwise condition, so
$\SE_{\bnd}(F)$ need not be the selection space of a soft set. For example, let
$A=\N$, $F(t)=\R$, and $a_t=\iota_t(t)$. Every $a_t$ is bounded, but their patch
$p(t)=t$ is not. By Theorem~\ref{thm:3.3},
$\ell^\infty(\N)$ is not of the form $\SE(H)$ for a soft set $H$ over $\R$.
\end{remark}

\begin{theorem}\label{thm:5.4}
The normed space $(\SE_{\bnd}(F),\lVert\cdot\rVert_\infty)$ is Banach if and only
if every section $(F(t),\lVert\cdot\rVert_t)$ is Banach.
\end{theorem}

\begin{proof}
Assume first that every section is Banach and let $(a_n)$ be Cauchy in the
supremum norm. For each $t$, the sequence $(a_n(t))$ is Cauchy in $F(t)$; write
$a_n(t)\to a(t)\in F(t)$. Given $\varepsilon>0$, choose $n_0$ such that
\[
 \sup_t\lVert a_m(t)-a_n(t)\rVert_t<\varepsilon/2
 \qquad(m,n\ge n_0).
\]
Fix $n\ge n_0$ and let $m\to\infty$ in each section. Then
$\lVert a(t)-a_n(t)\rVert_t\le\varepsilon/2$ for all $t$. Hence
$a-a_n\in\SE_{\bnd}(F)$. Since $a_n\in\SE_{\bnd}(F)$,
\[
        a=a_n+(a-a_n)\in\SE_{\bnd}(F),
\]
and $\lVert a-a_n\rVert_\infty\le\varepsilon/2<\varepsilon$.

Conversely, fix $t\in A$. The map
$\iota_t\colon F(t)\to\SE_{\bnd}(F)$ is linear and isometric. Its image is closed:
if $\iota_t(\alpha_n)\to a$ in the supremum norm, then $a(s)=0$ for $s\ne t$ and
$a=\iota_t(a(t))$. Thus $F(t)$ is isometric to a closed subspace of the Banach
space $\SE_{\bnd}(F)$ and is complete.
\end{proof}

\section{Convergence and the box, uniform, and product topologies}

\begin{definition}\label{def:6.1}
Let $(a_n)$ be a sequence in $\SE(F)$ and let $a\in\SE(F)$. We say that
$a_n\to a$
\begin{enumerate}[label=\textup{(\roman*)}]
\item \emph{componentwise} if
      $\lVert a_n(t)-a(t)\rVert_t\to0$ for every $t\in A$;
\item \emph{uniformly} if $d_\infty(a_n,a)\to0$;
\item \emph{softly} if for every strictly positive soft real $\varepsilon$ there
      is $n_0$ such that
      \[
       \snorm(a_n-a)\slt\varepsilon \qquad(n\ge n_0).
      \]
\end{enumerate}
\end{definition}

The condition $d_\infty(a_n,a)\to0$ allows $d_\infty(a_n,a)$ to be infinite
for finitely many $n$.

\begin{theorem}
For sequences in $\SE(F)$,
\[
  \text{soft convergence}\ \Longrightarrow\
  \text{uniform convergence}\ \Longrightarrow\
  \text{componentwise convergence}.
\]
The three modes agree for every sequence if and only if $A_0$ is finite. If
$A_0$ is infinite, both implications are strict.
\end{theorem}

\begin{proof}
If $a_n\to a$ softly, apply the definition to the constant soft real
$\overline{\delta/2}$, where $\delta>0$. Then eventually
$\lVert a_n(t)-a(t)\rVert_t<\delta/2$ for every $t$, so
$d_\infty(a_n,a)\le\delta/2<\delta$. Uniform convergence implies
componentwise convergence.

If $A_0$ is finite and $a_n\to a$ componentwise, let $\varepsilon$ be a strictly
positive soft real. For each $t\in A_0$, choose $n_t$ such that
$\lVert a_n(t)-a(t)\rVert_t<\varepsilon(t)$ for $n\ge n_t$. The maximum of the
finitely many $n_t$ works for all active coordinates; the remaining coordinates
are identically zero. Hence the convergence is soft.

Suppose now that $A_0$ is infinite. Choose distinct $t_1,t_2,\ldots\in A_0$ and
vectors $u_j\in F(t_j)$ with $\lVert u_j\rVert_{t_j}=1$. Define
\[
 x_n(t_j)=\frac1n u_j\quad(j\in\N),
 \qquad x_n(t)=0\quad(t\notin\{t_j:j\in\N\}).
\]
Then $d_\infty(x_n,\theta)=1/n$, so $x_n\to\theta$ uniformly. Let
$\varepsilon(t_j)=1/(j+1)$ and set $\varepsilon(t)=1$ elsewhere. For every $n$,
\[
 \lVert x_n(t_n)\rVert_{t_n}=\frac1n>\frac1{n+1}=\varepsilon(t_n),
\]
so the convergence is not soft.

Finally, let $y_n=\iota_{t_n}(u_n)$. Every fixed coordinate is eventually zero,
so $y_n\to\theta$ componentwise, but
$d_\infty(y_n,\theta)=1$ for every $n$. Thus componentwise convergence need not
be uniform.
\end{proof}

For $a\in\SE(F)$ and a strictly positive soft real $\varepsilon$, put
\[
 B_s(a,\varepsilon)
 :=\{b\in\SE(F):\snorm(b-a)\slt\varepsilon\}
 =\prod_{t\in A}B_t\bigl(a(t),\varepsilon(t)\bigr),
\]
where $B_t(\alpha,r)$ is the open ball in $F(t)$. Let $\taubox$ be the box
topology generated by the products $\prod_tU_t$, where each $U_t$ is open in
$F(t)$, and let $\tauprod$ be the product topology. For $r>0$, write
\[
 B_\infty(a,r)=\{b\in\SE(F):d_\infty(a,b)<r\}.
\]
The balls $B_\infty(a,r)$ generate the extended-metric topology $\tauunif$.
Equivalently, $\tauunif$ is the disjoint union of the metric topologies on the
components $a+\SE_{\bnd}(F)$.

\begin{proposition}
The soft balls $B_s(a,\varepsilon)$ form a basis for $\taubox$. A sequence
converges softly if and only if it converges in $\taubox$. Moreover, $\taubox$ is
Hausdorff.
\end{proposition}

\begin{proof}
Every soft ball is a box of sectionwise open balls. Conversely, if
$a\in\prod_tU_t$ and each $U_t$ is open, choose $\varepsilon(t)>0$ such that
$B_t(a(t),\varepsilon(t))\subseteq U_t$. Then
$B_s(a,\varepsilon)\subseteq\prod_tU_t$. Thus the soft balls form a basis for the
box topology. The two notions of sequential convergence now agree by
Definition~\ref{def:6.1}.

If $a\ne b$, choose $t$ with $a(t)\ne b(t)$ and take disjoint norm balls
around $a(t)$ and $b(t)$ in $F(t)$. Use the full section $F(s)$ at every
$s\ne t$. The resulting boxes are disjoint neighborhoods of $a$ and $b$.
\end{proof}

\begin{proposition}
One has
\[
          \tauprod\subseteq\tauunif\subseteq\taubox.
\]
If $A_0$ is finite, the three topologies coincide. If $A_0$ is infinite, both
inclusions are strict. In that case $\taubox$ is not first countable and hence is
not metrizable.
\end{proposition}

\begin{proof}
Let $a\in\SE(F)$. A basic product neighborhood of $a$ restricts only
finitely many coordinates. In each restricted coordinate, choose a norm ball
contained in the prescribed open set. A positive radius smaller than all these
finitely many radii gives a $d_\infty$-ball contained in the product
neighborhood. Hence $\tauprod\subseteq\tauunif$. To prove $\tauunif\subseteq\taubox$, let $b\in B_\infty(a,r)$ and put
\[
       s=\frac{r-d_\infty(a,b)}{2}>0.
\]
If $c\in\prod_tB_t(b(t),s)$, then
$d_\infty(c,b)\le s$ and hence
\[
 d_\infty(c,a)\le s+d_\infty(b,a)<r.
\]
Thus every point of $B_\infty(a,r)$ has a box neighborhood contained in that
ball.

If $A_0$ is finite, singleton factors outside $A_0$ have no effect and the three
topologies reduce to the usual topology of a finite product.

Now assume that $A_0$ is infinite, and choose distinct active parameters
$t_1,t_2,\ldots$. The uniform ball $B_\infty(\theta,1)$ is not product-open at
$\theta$: any basic product neighborhood leaves some active coordinate
unrestricted, where a vector of norm at least $1$ can be inserted. Thus the first
inclusion is strict.

Choose a strictly positive soft real $\varepsilon$ with
$\varepsilon(t_j)=1/j$ and $\varepsilon(t)=1$ elsewhere. The box
$B_s(\theta,\varepsilon)$ contains no uniform ball about $\theta$. Indeed, for
$r>0$, choose $j$ with $1/j<r$ and insert at $t_j$ a vector with norm strictly
between $1/j$ and $r$. Hence the second inclusion is strict.

For first countability, suppose that $(U_n)$ is a countable neighborhood basis
at $\theta$. Choose a soft ball $B_s(\theta,\varepsilon^n)\subseteq U_n$ for each
$n$. Define $\varepsilon(t_n)=\varepsilon^n(t_n)/2$ and set
$\varepsilon(t)=1$ at the remaining parameters. For each $n$, choose
$x_n\in F(t_n)$ with
$\lVert x_n\rVert_{t_n}=3\varepsilon^n(t_n)/4$. Then
\[
 \iota_{t_n}(x_n)\in B_s(\theta,\varepsilon^n)
 \setminus B_s(\theta,\varepsilon).
\]
Therefore no $U_n$ is contained in $B_s(\theta,\varepsilon)$, contradicting the
basis property.
\end{proof}

The classical product, uniform, and box topologies are discussed in
\cite[Section~20]{Munkres}.

\section{Completeness}

\begin{definition}
A sequence $(a_n)$ in $\SE(F)$ is
\begin{enumerate}[label=\textup{(\roman*)}]
\item \emph{componentwise Cauchy} if $(a_n(t))$ is Cauchy in $F(t)$ for every
      $t$;
\item \emph{uniformly Cauchy} if for every $\varepsilon>0$ there is $n_0$ such
      that $d_\infty(a_m,a_n)<\varepsilon$ for $m,n\ge n_0$;
\item \emph{soft Cauchy} if for every strictly positive soft real $\varepsilon$
      there is $n_0$ such that
      $\snorm(a_m-a_n)\slt\varepsilon$ for $m,n\ge n_0$.
\end{enumerate}
\end{definition}

\begin{theorem}\label{thm:7.2}
For a soft normed linear space $(F,\snorm)$, the following are equivalent.
\begin{enumerate}[label=\textup{(\roman*)}]
\item Every section $F(t)$ is Banach.
\item $\SE_{\bnd}(F)$ is Banach in the supremum norm.
\item Every uniformly Cauchy sequence in $\SE(F)$ converges uniformly.
\item Every componentwise Cauchy sequence in $\SE(F)$ converges componentwise.
\item Every soft Cauchy sequence in $\SE(F)$ converges softly.
\end{enumerate}
\end{theorem}

\begin{proof}
The equivalence of (i) and (ii) is Theorem~\ref{thm:5.4}.

Assume (i) and let $(a_n)$ be uniformly Cauchy. For each $t$, the sequence
$(a_n(t))$ has a limit $a(t)\in F(t)$. Given $\varepsilon>0$, choose $n_0$ such
that $d_\infty(a_m,a_n)<\varepsilon/2$ for $m,n\ge n_0$. Fix $n\ge n_0$ and let
$m\to\infty$ in each coordinate. Then
\[
 \lVert a(t)-a_n(t)\rVert_t\le\varepsilon/2
 \qquad(t\in A),
\]
so $d_\infty(a,a_n)\le\varepsilon/2<\varepsilon$. Thus (iii) holds.
Conversely, a Cauchy sequence $(\alpha_n)$ in $F(t)$ gives the uniformly Cauchy
sequence $(\iota_t(\alpha_n))$ in $\SE(F)$. A uniform limit yields a limit of
$(\alpha_n)$ in $F(t)$. Hence (iii) implies (i).

Taking coordinatewise limits proves that (i) implies (iv). For the
converse, apply (iv) to the supported selections $\iota_t(\alpha_n)$.

Assume (i) and let $(a_n)$ be soft Cauchy. Coordinatewise it is Cauchy, so define
$a(t)=\lim_na_n(t)$. Given a strictly positive soft real $\varepsilon$, choose
$n_0$ such that
\[
 \lVert a_m(t)-a_n(t)\rVert_t<\varepsilon(t)/2
 \qquad(t\in A,\ m,n\ge n_0).
\]
Fix $n\ge n_0$ and let $m\to\infty$ coordinatewise. Then
$\lVert a(t)-a_n(t)\rVert_t\le\varepsilon(t)/2<\varepsilon(t)$ for every $t$,
which proves soft convergence. Thus (v) holds. Conversely, applying (v) to
$\iota_t(\alpha_n)$ proves completeness of $F(t)$.
\end{proof}

\begin{definition}
A soft normed linear space satisfying the conditions of
Theorem~\ref{thm:7.2} is called a \emph{soft Banach space via soft
elements}.
\end{definition}

\begin{remark}
The three Cauchy conditions in Theorem~\ref{thm:7.2} are not
equivalent. On $\R^{\N}$, the constant selections $a_n(t)=1/n$ are uniformly
Cauchy but not soft Cauchy: for $\varepsilon(t)=1/(t+1)$ and any tail, choose
distinct $m,n$ in the tail and then $t$ so large that
$\varepsilon(t)<|1/m-1/n|$. The moving bumps $\iota_n(1)$ are componentwise
Cauchy but not uniformly Cauchy.
\end{remark}

\section{Soft linear operators}

Let $(F,\snorm_F)$ and $(G,\snorm_G)$ be soft normed linear spaces over the same
parameter set $A$. Write $\lVert\cdot\rVert_{F,t}$ and
$\lVert\cdot\rVert_{G,t}$ for the corresponding norms on $F(t)$ and $G(t)$.

\begin{definition}[Soft linear operator]
A map $T\colon\SE(F)\to\SE(G)$ is \emph{soft linear} if
\[
 T(a+b)=T(a)+T(b),\qquad T(ra)=rT(a)
\]
for all $a,b\in\SE(F)$ and every soft scalar $r\in\KA$.
\end{definition}

\begin{theorem}
A map $T\colon\SE(F)\to\SE(G)$ is soft linear if and only if there is a unique
family of linear maps $T_t\colon F(t)\to G(t)$ such that
\[
          T(a)(t)=T_t(a(t))
          \qquad(a\in\SE(F),\ t\in A).
\]
\end{theorem}

\begin{proof}
A sectionwise family of linear maps defines a soft linear map. Conversely,
let $T$ be soft linear. Since $e_ta=\iota_t(a(t))$,
\[
 e_tT(a)=T(e_ta)=T\bigl(\iota_t(a(t))\bigr).
\]
Evaluating at $t$ shows that $T(a)(t)$ depends only on $a(t)$. Define
\[
       T_t(\alpha):=T(\iota_t(\alpha))(t).
\]
Then $T(a)(t)=T_t(a(t))$. Additivity and homogeneity for constant scalars show
that $T_t$ is linear, and the formula gives uniqueness.
\end{proof}

\begin{corollary}
For every $t\in A$ and $\alpha\in F(t)$, the set
\[
 \{T(a)(t):a\in\SE(F),\ a(t)=\alpha\}
\]
is a singleton.
\end{corollary}

\begin{proof}
By the sectionwise formula in the preceding theorem, the displayed set is
$\{T_t(\alpha)\}$. This is condition (L3) in
\cite[Theorem~3.10]{DSoperators}, so (L3) follows from soft linearity.
\end{proof}

\begin{example}
On $\R^2$ with $A=\{1,2\}$, the map
$T(a)=(a(1)+a(2),a(2))$ is linear over the constant scalars, but it is not soft
linear because $T(e_1a)\ne e_1T(a)$ when $a(2)\ne0$. This is the operator
version of Example~\ref{ex:4.3}.
\end{example}

\begin{definition}
A soft linear operator $T$ is \emph{soft bounded} if there is a nonnegative soft
real $M\in\R^{A}$ such that
\[
      \snorm_G(Ta)\sle M\,\snorm_F(a)
      \qquad(a\in\SE(F)).
\]
\end{definition}

\begin{proposition}\label{prop:8.6}
Let $T$ have section operators $T_t$. Then $T$ is soft bounded if and only if
every $T_t$ is bounded. In this case
\[
       \lVert T\rVert_s(t):=\lVert T_t\rVert \qquad(t\in A)
\]
is the least soft bound.
\end{proposition}

\begin{proof}
If every $T_t$ is bounded, then
\[
 \lVert T(a)(t)\rVert_{G,t}\le\lVert T_t\rVert\,\lVert a(t)\rVert_{F,t}
\]
for all $a$ and $t$. Conversely, apply a soft bound to
$\iota_t(\alpha)$ and evaluate at $t$. This gives
$\lVert T_t\alpha\rVert_{G,t}\le M(t)\lVert\alpha\rVert_{F,t}$, so
$\lVert T_t\rVert\le M(t)$.
\end{proof}

A soft linear operator is called \emph{soft sequentially continuous} if
$a_n\to a$ softly implies $T(a_n)\to T(a)$ softly.

\begin{theorem}
For a soft linear operator $T$, the following are equivalent.
\begin{enumerate}[label=\textup{(\roman*)}]
\item Every section operator $T_t$ is bounded.
\item $T$ is soft bounded.
\item $T\colon(\SE(F),\taubox)\to(\SE(G),\taubox)$ is continuous.
\item $T$ is soft sequentially continuous.
\end{enumerate}
\end{theorem}

\begin{proof}
The equivalence of (i) and (ii) is Proposition~\ref{prop:8.6}. If (i)
holds, the inverse image of a basic box $\prod_tV_t$ is
$\prod_tT_t^{-1}(V_t)$, which is box-open. Hence (iii) holds, and (iii) implies
(iv).

Assume (iv) and suppose that $T_t$ is unbounded for some $t$. For each
$n\in\N$, choose $u_n\in F(t)$ such that
\[
       \lVert u_n\rVert_{F,t}=1,
       \qquad \lVert T_tu_n\rVert_{G,t}\ge n.
\]
Set $\alpha_n=u_n/n$. Then $\alpha_n\to0$ in $F(t)$, but
$\lVert T_t\alpha_n\rVert_{G,t}\ge1$. The supported selections
$\iota_t(\alpha_n)$ converge softly to $\theta$, while their images do not
converge even componentwise to $\theta$. This contradicts (iv).
\end{proof}

\begin{proposition}\label{prop:8.8}
Let $T$ be soft linear and put
\[
       q:=\sup_{t\in A}\lVert T_t\rVert\in[0,\infty],
\]
where $\lVert T_t\rVert=\infty$ when $T_t$ is unbounded. The following are
equivalent.
\begin{enumerate}[label=\textup{(\roman*)}]
\item $q<\infty$.
\item There is a constant $C<\infty$ such that
      \[
        d_\infty(Ta,Tb)\le C\,d_\infty(a,b)
      \]
      whenever $d_\infty(a,b)<\infty$; equivalently, $T$ is Lipschitz on every
      finite-distance component with one common constant.
\item $T\colon(\SE(F),\tauunif)\to(\SE(G),\tauunif)$ is continuous.
\item $T$ maps $\SE_{\bnd}(F)$ into $\SE_{\bnd}(G)$ and its restriction is a
      bounded linear operator.
\end{enumerate}
When these conditions hold, the least admissible constant in (ii) and the norm of
the restriction in (iv) are both equal to $q$.
\end{proposition}

\begin{proof}
If (i) holds, taking suprema of the sectionwise estimates proves (ii) with $C=q$.
Condition (ii) gives continuity on each finite-distance component and hence
continuity for $\tauunif$. Suppose (iii) holds. Continuity at $\theta$ gives
$\delta>0$ such that
$d_\infty(a,\theta)<\delta$ implies $d_\infty(Ta,\theta)<1$.
For a nonzero section $F(t)$ and $\alpha\in F(t)$ with
$\lVert\alpha\rVert_{F,t}=1$, apply this to
$a=(\delta/2)\iota_t(\alpha)$. Then
$\lVert T_t\alpha\rVert_{G,t}<2/\delta$, independently of $t$. The zero sections
have operator norm zero, so $q\le2/\delta<\infty$.

If (i) holds, then
$\lVert Ta\rVert_\infty\le q\lVert a\rVert_\infty$, so (iv) holds with operator
norm at most $q$. Conversely, if the restriction has norm $C$, applying it to
$\iota_t(\alpha)$ gives $\lVert T_t\rVert\le C$ for every $t$. Thus $q\le C$.
The same supported selections show that every constant in (ii) is at least $q$.
Hence both optimal constants are equal to $q$.
\end{proof}

\begin{corollary}
Assume every $T_t$ is bounded and
$T(\SE_{\bnd}(F))\subseteq\SE_{\bnd}(G)$. Then
$\sup_t\lVert T_t\rVert<\infty$, and the restriction of $T$ is bounded.
\end{corollary}

\begin{proof}
Suppose that the section norms are not uniformly bounded. Since every
$\lVert T_t\rVert$ is finite, choose distinct parameters $t_n$ such that
$\lVert T_{t_n}\rVert>n$. For each $n$, choose $\alpha_n\in F(t_n)$ with
\[
 \lVert\alpha_n\rVert_{F,t_n}=1,
 \qquad
 \lVert T_{t_n}\alpha_n\rVert_{G,t_n}>n.
\]
Patch the $\alpha_n$ at the coordinates $t_n$ and use zero elsewhere. The
resulting selection is bounded, but its image is not. This contradicts the
assumption.
\end{proof}

\begin{example}
Let $A=\N$, $F(t)=G(t)=\R$, and $T_t=t\,I$. Every section operator is bounded,
so $T$ is soft bounded and box-continuous. However,
$\sup_t\lVert T_t\rVert=\infty$, and $T$ does not map $\ell^\infty$ into itself:
the constant selection $1$ is sent to $(t)_{t\in\N}$.
\end{example}

\section{Applications}

\subsection{A Neumann-series theorem}

Let $(F,\snorm)$ be a soft Banach space. Let $I$ denote the identity on
$\SE(F)$ and $I_t$ the identity on $F(t)$.

\begin{theorem}
Let $S\colon\SE(F)\to\SE(F)$ be soft linear and suppose
\[
          q:=\sup_{t\in A}\lVert S_t\rVert<1.
\]
Then $I-S$ is a bijection of $\SE(F)$ and
\[
 (I-S)^{-1}(b)(t)=(I_t-S_t)^{-1}b(t),
 \qquad
 \snorm\bigl((I-S)^{-1}b\bigr)
 \sle \frac{1}{1-q}\,\snorm(b).
\]
Moreover, $I-S$ restricts to a bounded automorphism of $\SE_{\bnd}(F)$, and
\[
       (I-S)^{-1}=\sum_{n=0}^{\infty}S^n
\]
with convergence in the operator norm on $\SE_{\bnd}(F)$.
\end{theorem}

\begin{proof}
For each $t\in A$, the series
\[
          R_t:=\sum_{n=0}^{\infty}S_t^n
\]
converges in operator norm on the Banach space $F(t)$ because
$\lVert S_t\rVert\le q<1$. It satisfies
\[
 (I_t-S_t)R_t=R_t(I_t-S_t)=I_t,
 \qquad
 \lVert R_t\rVert\le\frac{1}{1-q}.
\]
Define $R\colon\SE(F)\to\SE(F)$ by
\[
             R(b)(t)=R_tb(t).
\]
Then, coordinatewise,
\[
          (I-S)R=I,\qquad R(I-S)=I.
\]
Thus $R=(I-S)^{-1}$, and the estimate for $R_t$ gives the stated soft bound.

By Proposition~\ref{prop:8.8}, the restriction of $S$ to
$\SE_{\bnd}(F)$ has norm at most $q$. For $N\ge m$,
\[
 \left\lVert\sum_{n=m}^{N}S^n\right\rVert
 \le\sum_{n=m}^{N}q^n
 \le\frac{q^m}{1-q}.
\]
Hence the series converges in operator norm on $\SE_{\bnd}(F)$. Its limit is the
restriction of $R$, so $I-S$ is a bounded automorphism of
$\SE_{\bnd}(F)$.
\end{proof}

\begin{proposition}
Suppose $\lVert S_t\rVert<1$ for every $t$, with no common bound below one. Then
$I-S$ is a bijection of $\SE(F)$ and its inverse is soft bounded, with
\[
 \lVert(I_t-S_t)^{-1}\rVert\le\frac{1}{1-\lVert S_t\rVert}.
\]
The restriction of $I-S$ to $\SE_{\bnd}(F)$ need not be onto.
\end{proposition}

\begin{proof}
Apply the classical Neumann theorem in each section and patch the inverses.
For the failure on bounded selections, take
$A=\N=\{1,2,\ldots\}$, $F(t)=\R$, and
\[
       S_t=\left(1-\frac1{t+1}\right)I.
\]
Then $(I-S)a(t)=a(t)/(t+1)$. On $\ell^\infty$, this is an injective bounded
operator of norm
\[
       \sup_{t\ge1}\frac1{t+1}=\frac12,
\]
but it is not onto: the constant selection $1$ has the unique preimage
$a(t)=t+1$, which is unbounded.
\end{proof}

\subsection{Sectionwise contractions and the mode of convergence}

\begin{theorem}
Let $(F,\snorm)$ be a soft Banach space and let
$T\colon\SE(F)\to\SE(F)$ be sectionwise:
$T(a)(t)=T_t(a(t))$. Suppose each $T_t$ is a contraction with constant
$k(t)<1$.
\begin{enumerate}[label=\textup{(\alph*)}]
\item $T$ has a unique fixed point $a^*\in\SE(F)$, and every Picard iteration
      $a_{n+1}=T(a_n)$ converges to $a^*$ componentwise.
\item If $q:=\sup_tk(t)<1$ and
      \[
       D:=\sup_{t\in A}\lVert a_0(t)-T_t(a_0(t))\rVert_t<\infty,
      \]
      then
      \[
       d_\infty(a_n,a^*)\le\frac{q^nD}{1-q},
      \]
      so the convergence is uniform.
\end{enumerate}
\end{theorem}

\begin{proof}
Banach's contraction theorem applies to each $T_t$; see
\cite{Banach1922}. Hence $T_t$ has a unique fixed point $\alpha_t^*$, and
$T_t^n(a_0(t))\to\alpha_t^*$. The patch $a^*(t)=\alpha_t^*$ is the unique
fixed point of $T$, and the convergence is componentwise.

For (b), the classical contraction estimate gives
\[
 \lVert a_0(t)-\alpha_t^*\rVert_t
 \le \frac{\lVert a_0(t)-T_t(a_0(t))\rVert_t}{1-k(t)}
 \le\frac{D}{1-q}.
\]
Therefore
\[
 \lVert a_n(t)-\alpha_t^*\rVert_t
 \le k(t)^n\lVert a_0(t)-\alpha_t^*\rVert_t
 \le\frac{q^nD}{1-q}.
\]
Taking the supremum over $t$ proves the result.
\end{proof}

\begin{example}[Uniform convergence need not be soft]
Let $A=\N$, $F(t)=\R$, fix $0<q<1$, and set
$T_t(x)=qx+(1-q)$. Starting from $a_0=\theta$, one has
$a_n(t)=1-q^n$ and the fixed point is $a^*=\bar1$. Thus
$d_\infty(a_n,a^*)=q^n\to0$. Define the strictly positive soft real
$\varepsilon(t)=q^t$. For every $n\ge0$, choose $t=n+1$; then
\[
 \lVert a_n(t)-a^*(t)\rVert=q^n>q^{n+1}=\varepsilon(t).
\]
Thus the Picard sequence converges geometrically in $d_\infty$ but does not
converge softly.
\end{example}

\section{Conclusion}

Soft norms and soft linear operators on $\SE(F)$ decompose into families on the
sections $F(t)$. The box, extended-supremum, and product structures distinguish
soft, uniform, and componentwise behavior, while common section bounds control
uniform operator results. It remains open whether the coordinate-idempotent
method extends to soft locally convex spaces defined by families of soft
seminorms.

\section*{Conflicts of Interest} \medskip
The author declares that there are no conflicts of interest regarding the
publication of this article.

\section*{Authors Contribution Statement} \medskip
The author conceived the study, proved the results, prepared the manuscript,
and approved the final version.

\end{document}